\documentclass[reqno]{amsart}

\usepackage[utf8]{inputenc}
\usepackage[english]{babel}

\usepackage{fullpage}

\usepackage{amsmath}
\usepackage{amsfonts}
\usepackage{amssymb}

\DeclareMathOperator{\Int}{Int}
\DeclareMathOperator{\adj}{adj}

\usepackage{colonequals}

\usepackage{graphicx}
\newcommand\spikey{\raisebox{-\dp\strutbox}{\includegraphics[height=\dimexpr\dp\strutbox+\ht\strutbox]{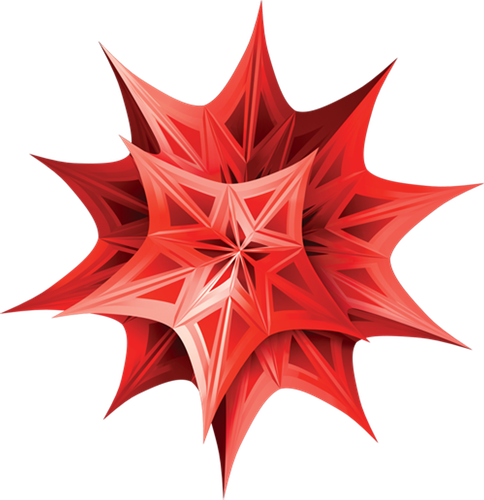}}}
\newcommand\fox{\includegraphics[height=\ht\strutbox]{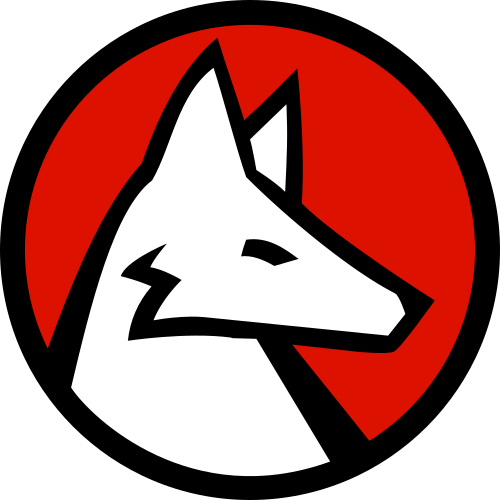}}
\makeatletter
\newcommand\WM[1]{\ifst@rred\tag*{\hspace*\@}\else\stepcounter{equation}\tag{\theequation}\fi
\g@addto@macro\df@tag{\llap{\rlap{\spikey~\textup{\ref{#1}}}\hspace*\linewidth}}}
\makeatother

\def\nth#1{\ifmmode#1\else$#1$\fi\textsuperscript{th}}

\usepackage{enumitem}
\usepackage[english=british]{csquotes}
\usepackage{url}

\DeclareUnicodeCharacter{00AE}{\raisebox{1ex}{\scalebox{.5}{\textregistered}}}

\usepackage{amsthm}
\usepackage{theoremref}

\theoremstyle{plain}

\newtheorem{theorem}{Theorem}[section]
\newtheorem{lemma}[theorem]{Lemma}
\newtheorem{corollary}[theorem]{Corollary}
\newtheorem{proposition}[theorem]{Proposition}
\newtheorem{assertion}[theorem]{Assertion}

\theoremstyle{definition}

\newtheorem{definition}[theorem]{Definition}
\newtheorem*{notation}{Notation}

\theoremstyle{remark}

\newtheorem*{remark}{Remark}
\pretocmd{\endremark}{\hfill$\blacktriangleleft$}{}{}

\makeatletter
\def\@cite#1#2{[\textbf{#1}\if@tempswa , #2\fi]}
\def\@biblabel#1{[\textbf{#1}]}
\makeatother

\begin{document}

\title{Differential characterization of quadratic surfaces}

\author[B. Zawalski]{Bartłomiej Zawalski}
\address{Institute of Mathematics, Polish Academy of Sciences}
\email{b.zawalski@impan.pl}
\thanks{This work was supported by the Polish National Science Centre grants nos. 2015/18/A/ST1/00553 and 2020/02/Y/ST1/00072. Numerical computations were performed within the Interdisciplinary Centre for Mathematical and Computational Modelling UW grant no. G76-28.}

\begin{abstract}
Let $f\in W^{3,1}_{\mathrm{loc}}(\Omega)$ be a function defined on a connected open subset $\Omega\subseteq\mathbb R^2$. We will show that its graph is contained in a quadratic surface if and only if $f$ is a weak solution to a certain system of third-order partial differential equations unless the Hessian determinant of $f$ is non-positive everywhere on $\Omega$. Moreover, we will prove that the system is, in a sense, the simplest possible in a wide class of differential equations, which will lead to the classification of all polynomial partial differential equations satisfied by parametrizations of generic quadratic surfaces. Although we will mainly use the tools of linear and commutative algebra, the theorem itself is also somewhat related to holomorphic functions.
\end{abstract}

\maketitle

\paragraph{\bf Acknowledgement}

This preprint has not undergone peer review or any post-submission improvements or corrections. The Version of Record of this article is published in Beitr\"age zur Algebra und Geometrie / Contributions to Algebra and Geometry, and is available online at \url{https://doi.org/10.1007/s13366-024-00735-0}. This arXiv version contains a strengthened form of \thref{cor:01}, not included in the published journal version.

\tableofcontents

\section{Introduction}

It was already known since the works of Blaschke \cite[p.~18]{blaschke1923vorlesungen} that conics are the only planar curves with constant equiaffine curvature. In a special case of a graph of a function of class $C^5$, this condition is equivalent to a certain fifth-order ordinary differential equation, which reads
\begin{equation}\label{eq:07}9f''(x)^2f^{(5)}(x)-45f''(x)f^{(3)}(x)f^{(4)}(x)+40f^{(3)}(x)^3=0.\end{equation}
In higher dimensions, hyperquadrics are characterized by the Maschke-Pick-Berwald theorem \cite[Theorem~4.5]{nomizu1994affine} as the only hypersurfaces with vanishing cubic form $C$ defined in \cite[ch.~II, s.~4]{nomizu1994affine}. However, the definition implicitly uses the intrinsic Blaschke structure and thus the cubic form $C$ can hardly be expressed in an extrinsic coordinate system. It is also unclear what minimal smoothness we need to assume. Nevertheless, such a result for $2$-dimensional surfaces turns out to be a consequence of two relatively simple partial differential equations. The aim of this paper is to prove the following main theorem:

\begin{theorem}\thlabel{thm:01}
Let $f\in W^{3,1}_{\mathrm{loc}}(\Omega)$ be a function from the local Sobolev space\footnote{Fix $1\leq p\leq\infty$ and let $k\in\mathbb N$. The local Sobolev space $W^{k,p}_{\mathrm{loc}}(\Omega)$ consists of all locally integrable functions $f:\Omega\to\mathbb R$ such that for every multi-index $\boldsymbol\alpha$ with $|\boldsymbol\alpha|\leq k$, the mixed partial derivative $f^{(\boldsymbol\alpha)}$ exists in the weak sense and belongs to $L^p_{\mathrm{loc}}(\Omega)$ (cf. \cite[s.~5.2.2]{evans2010partial}).}, defined on a connected open subset $\Omega\subseteq\mathbb R^2$. Suppose that the Hessian determinant of $f$ is somewhere positive. Then $f$ is a weak solution to the system of partial differential equations
\begin{equation}\label{eq:05}
\begin{aligned}
f^{(3,0)} f^{(0,2) 2}-3 f^{(1,2)} f^{(2,0)} f^{(0,2)}+2 f^{(0,3)} f^{(1,1)} f^{(2,0)} &= 0, \\
f^{(0,3)} f^{(2,0) 2}-3 f^{(2,1)} f^{(0,2)} f^{(2,0)}+2 f^{(3,0)} f^{(1,1)} f^{(0,2)} &= 0
\end{aligned}
\end{equation}
if and only if its graph is contained in a quadratic surface.
\end{theorem}

Therefore \thref{thm:01} can be considered a $2$-dimensional analog of the aforementioned result of Blaschke. Contrary to Maschke-Pick-Berwald, it is formulated in terms of simple, explicit partial differential equations, with weaker smoothness assumptions. Moreover, we will show that the system \eqref{eq:05} is minimal in the sense that the left-hand sides form a minimal generating set (viz. a reduced Gr\"obner basis) of a certain differential ideal.\\

Such a characterization of quadratic surfaces of positive Gaussian curvature as the only solutions to certain partial differential equations without boundary conditions may be useful when one wants to prove that some specific convex body is an ellipsoid using, e.g., the tools of differential geometry. Such problems arise naturally in convex geometry, especially in various characterizations of Hilbert spaces among all finite-dimensional Banach spaces.\\

As superfluous as it may seem, the assumption on the Hessian determinant is not purely technical, as the following holds:

\begin{theorem}\thlabel{thm:02}
Let $f\in W^{3,1}_{\mathrm{loc}}(\Omega)$ be a function from the local Sobolev space, defined on a connected open subset $\Omega\subseteq\mathbb R^2$. Suppose that the Hessian determinant of $f$ is non-positive. Then $f$ is a weak solution to the system of partial differential equations \eqref{eq:05} if and only if $\Omega$ contains a countable sum of disjoint open connected subsets $\Omega_i$ such that:
\begin{enumerate}[label=\textup{(\makebox[0.625em]{\arabic*})}]
\item on each $\Omega_i$ the graph of $f$ is contained in either:
\begin{enumerate}[label=\textup{(\makebox[0.625em]{\alph*})}]
\item a doubly-ruled surface\footnote{A ruled surface that contains two families of rulings.}, or
\item a developable surface\footnote{A ruled surface having Gaussian curvature $K=0$ everywhere.}, or
\item a Catalan surface\footnote{A ruled surface all of whose rulings are parallel to a fixed plane, called the directrix plane of the surface.} with directrix plane $XZ$, or
\item a Catalan surface with directrix plane $YZ$,
\end{enumerate}
\item the union $\bigcup\Omega_i$ is dense in $\Omega$.
\end{enumerate}
\end{theorem}

Note that all of the above are particular examples of ruled surfaces\footnote{A surface that can be swept out by moving a line in space. The straight lines themselves are called \emph{rulings}. The Gaussian curvature on a ruled regular surface is everywhere non-positive.}. Regrettably, the exact classification of solutions to \eqref{eq:05} seems to be a tedious, technical task and therefore will not be given here, so as not to overshadow the main idea.\\

To perform lengthy computations, we will employ a widely used technical computing system \textsf{Wolfram Mathematica} \cite{Mathematica}. Nevertheless, they still could have been done with pen and paper (albeit with some difficulty), and hence the proof remains human-surveyable. A thorough discussion of this aspect can be found in Appendix \ref{app:00}. For transparency, all the results obtained with the help of a computer are tagged with \enquote*{Spikey} (\spikey), followed by a reference to the relevant section of Appendix \ref{app:00}.

\section{Notation and basic concepts}

To prove the \thref{thm:01}, we will need some very general facts concerning quadratic surfaces, which are in themselves quite interesting. We begin with rephrasing the problem in the language of commutative algebra.

\begin{definition}
Let
$$R\colonequals\mathbb R\left[x,y,\partial^{(0,0)},\partial^{(0,1)},\partial^{(1,0)},\ldots\right]$$
be a ring of polynomials in variables $x$, $y$ and formal partial derivatives $\partial^{(i,j)}$ and let
$$S\colonequals\left\langle\partial^{(0,2)}, \partial^{(2,0)}, \partial^{(0,2)} \partial^{(2,0)}-\partial^{(1,1) 2}\right\rangle$$
be a submonoid of the multiplicative monoid of $R$, with the listed generators. By $S^{-1}R$ we denote a localisation of $R$ at $S$ \cite[s.~2.1]{eisenbud1995commutative}.
\end{definition}

The ring $S^{-1}R$ can be viewed as an algebra of a certain type of differential operators $T$ defined for those smooth functions $f:\mathbb R^2\supseteq\Omega\to\mathbb R$ for which all the expressions
\begin{equation}\label{eq:12}f^{(0,2)}(x,y),\quad f^{(2,0)}(x,y),\quad f^{(0,2)}(x,y) f^{(2,0)}(x,y)- f^{(1,1)}(x,y)^2\end{equation}
do not take a zero value on $\Omega$ and thus have reciprocals. We will call such functions \emph{generic}. Examples include but are not limited to functions with positive Hessian determinant, i.e., whose graphs have positive Gaussian curvature.

\begin{notation}
Let $\Omega\subseteq\mathbb R^2$ be a connected open subset of $\mathbb R^2$. Denote by $\mathcal G(\Omega)$ the set of generic functions $f:\Omega\to\mathbb R$ and by $\mathcal Q(\Omega)$ its subset consisting of parametrizations of quadratic surfaces.
\end{notation}

\begin{definition}
Let $D_x,D_y:S^{-1}R\to S^{-1}R$ be \emph{derivations} \cite[ch.~16]{eisenbud1995commutative}, i.e., $\mathbb R$-linear endomorphisms of additive group of $S^{-1}R$ satisfying the Leibniz product rule
$$D(r_1r_2)=D(r_1)r_2+r_1D(r_2),\quad r_1,r_2\in S^{-1}R,$$
and thus uniquely determined by their values on indeterminates:
\begin{gather*}
D_x(x)\colonequals 1,\quad D_x(y)\colonequals 0,\quad D_x(\partial^{(i,j)})\colonequals\partial^{(i+1,j)},\\
D_y(x)\colonequals 0,\quad D_y(y)\colonequals 1,\quad D_y(\partial^{(i,j)})\colonequals\partial^{(i,j+1)}.
\end{gather*}
In particular, the well-known formula for differentiating fractions
$$D\left(\frac{r}{s}\right)=\frac{D(r)s-rD(s)}{s^2}$$
follows from the Leibniz product rule. A ring $S^{-1}R$ equipped with derivations $D_x,D_y$ forms a \emph{differential ring}.
\end{definition}

\begin{definition}
A \emph{differential ideal} $\mathfrak a$ in a differential ring $R$ is an ideal that is mapped to itself by each derivation.
\end{definition}

\begin{definition}
Let $X$ be a subset of $\mathcal G(\Omega)$. The \emph{annihilator} of $X$ in $S^{-1}R$, denoted by $X^\dagger$, is a collection of differential operators $T\in S^{-1}R$ such that $Tf=0$ for all $f\in X$. The annihilator of any subset is clearly a differential ideal. The annihilator of an empty set is the whole $S^{-1}R$ and the annihilator of the whole $\mathcal G(\Omega)$ is just the zero operator.
\end{definition}

\section{Polynomial PDEs satisfied by generic quadratic surfaces}

Observe that a graph of a function $f$ is contained in a quadratic surface if and only if its every point satisfies a quadratic equation
\begin{equation}\label{eq:38}a_{11} x^2+a_{12} x y+a_{13} x f+a_{22} y^2+a_{23} y f+a_{33} f^2+b_1 x+b_2 y+b_3 f+c = 0\end{equation}
with constant coefficients $a_{ij},b_k,c$. This is equivalent to the fact that the set of functions
\begin{equation}\label{eq:01}\begin{Bmatrix} x^2, & x y, & x f(x,y), & y^2, & y f(x,y), & f(x,y)^2, & x, & y, & f(x,y), & 1 \end{Bmatrix}\end{equation}
is linearly dependent. That is how the concept of a generalized Wronskian for functions of several variables enters play. For clarity, we adopt the notation from \cite{WOLSSON198973}.

\begin{definition}[{\cite[Definition~1]{WOLSSON198973}}]
A \emph{generalised Wronskian} of $\boldsymbol\phi=(\phi_1(t),\ldots,\phi_n(t))$ is any determinant of the type
$$\begin{vmatrix} \boldsymbol\phi \\ \partial^1\boldsymbol\phi \\ \vdots \\ \partial^{n-1}\boldsymbol\phi \end{vmatrix},$$
where $\boldsymbol\phi$, $\partial^i\boldsymbol\phi$ are row vectors, $\partial^i$ is any partial derivative of order not greater that $i$ and all $\partial^i$ are distinct.
\end{definition}

\begin{remark}
Note that in the realm of functions in $m\geq 2$ variables, a generalized Wronskian of $\varphi$ is no longer unique, since there are many possible ways of choosing row vectors $\partial^i\boldsymbol\phi$ satisfying all the imposed conditions. More precisely, there are
$$\binom{m+i}{m}$$
partial derivatives of order not greater than $i$, and hence there are exactly
$$\prod_{i=0}^{n-1}\left(\binom{m+i}{m}-i\right)$$
generalised Wronskians of $n$ functions in $m$ variables. However, from now on, we will identify all generalized Wronskians that differ only in the order of rows, as it does not affect the rank of the matrix.
\end{remark}

\begin{notation}
Denote by $\boldsymbol\phi$ the tuple of functions \eqref{eq:01}.
\end{notation}

\begin{assertion}\thlabel{ass:01}
Each generalized Wronskian of $\boldsymbol\phi$ can be viewed as an element of $S^{-1}R$. Moreover, by the very definition, it belongs to $\mathcal Q(\Omega)^\dagger$. Indeed, if the set of functions \eqref{eq:01} is linearly dependent, then all its generalized Wronskians vanish identically since their columns are themselves linearly dependent.
\end{assertion}

The following key proposition characterizes the set of polynomial differential equations satisfied by the parametrization of any generic quadratic surface.

\begin{proposition}\thlabel{pro:01}
Let $\Omega\subseteq\mathbb R^2$ be a connected open subset of $\mathbb R^2$. Then the annihilator $\mathcal Q(\Omega)^\dagger\subseteq S^{-1}R$ is a differential ideal generated by
\begin{equation}\label{eq:09}
\begin{aligned}
&\partial^{(3,0)}\partial^{(0,2) 2}-3 \partial^{(1,2)} \partial^{(2,0)} \partial^{(0,2)}+2 \partial^{(0,3)} \partial^{(1,1)} \partial^{(2,0)}, \\
&\partial^{(0,3)}\partial^{(2,0) 2}-3 \partial^{(2,1)} \partial^{(0,2)} \partial^{(2,0)}+2 \partial^{(3,0)} \partial^{(1,1)} \partial^{(0,2)}.
\end{aligned}
\end{equation}
\end{proposition}

\begin{proof}
Clearly $\mathcal Q(\Omega)^\dagger$ is a differential ideal in $S^{-1}R$. Denote by $\mathfrak a$ the differential ideal generated by \eqref{eq:09}. We have to show that $\mathcal Q(\Omega)^\dagger=\mathfrak a$. We will do this by proving both inclusions.\\

First, we will show a simpler inclusion $\mathcal Q(\Omega)^\dagger\supseteq\mathfrak a$. Since both $\mathcal Q(\Omega)^\dagger$ and $\mathfrak a$ are differential ideals, it is enough to prove that the generators of $\mathfrak a$ are contained in $\mathcal Q(\Omega)^\dagger$. Let $f\in\mathcal Q(\Omega)$ be a parametrization of some generic quadratic surface. By \thref{ass:01}, all the generalized Wronskians of $\boldsymbol\phi$ vanish identically on $\Omega$. Denote by $W_{i,j}$ the generalised Wronskian of $\boldsymbol\phi$ formed by deleting the row $\boldsymbol\phi^{(i,j)}$ from
$$\begin{pmatrix} \boldsymbol\phi \\ \boldsymbol\phi^{(0,1)} \\ \boldsymbol\phi^{(1,0)} \\ \boldsymbol\phi^{(0,2)} \\ \boldsymbol\phi^{(1,1)} \\ \boldsymbol\phi^{(2,0)} \\ \boldsymbol\phi^{(0,3)} \\ \boldsymbol\phi^{(1,2)} \\ \boldsymbol\phi^{(2,1)} \\ \boldsymbol\phi^{(3,0)} \\ \boldsymbol\phi^{(0,4)} \end{pmatrix}.$$
The only non-trivial (i.e., not vanishing algebraically) ones are the following:
\begin{align*}
W_{3,0} &= 24 f^{(0,2) 2} \Big(3 f^{(2,1)} f^{(0,2) 2}-6 f^{(1,1)} f^{(1,2)} f^{(0,2)}- f^{(0,3)} f^{(2,0)} f^{(0,2)}+4 f^{(0,3)} f^{(1,1) 2}\Big) \\
W_{2,1} &= 72 f^{(0,2) 2} \Big(f^{(3,0)} f^{(0,2) 2}-3 f^{(1,2)} f^{(2,0)} f^{(0,2)}+2 f^{(0,3)} f^{(1,1)} f^{(2,0)}\Big) \\
W_{1,2} &= 72 f^{(0,2) 2} \Big(f^{(0,3)} f^{(2,0) 2}-3 f^{(0,2)} f^{(2,1)} f^{(2,0)}+2 f^{(0,2)} f^{(1,1)} f^{(3,0)}\Big) \\
W_{0,3} &= 24 f^{(0,2) 2} \Big(4 f^{(3,0)} f^{(1,1) 2}-6 f^{(2,0)} f^{(2,1)} f^{(1,1)}+3 f^{(1,2)} f^{(2,0) 2}-f^{(0,2)} f^{(2,0)} f^{(3,0)}\Big)
\end{align*}
Note that although the underlying matrices depend on fourth-order partial derivatives, their determinants do not, which is somehow intriguing.

\begin{remark}
Since $\boldsymbol\phi$ consists of $n=10$ functions and in $m=2$ variables, there are exactly $10$ partial derivatives of order at most $3$, there is a unique generalized Wronskian of $\boldsymbol\phi$ using partial derivatives of order at most $3$, namely $W_{0,4}$. However, it turns out that $\boldsymbol\phi^{(3,0)},\boldsymbol\phi^{(2,1)},\boldsymbol\phi^{(1,2)},\boldsymbol\phi^{(0,3)}$ are always linearly dependent. Indeed, observe that the $4\times 10$ matrix
\begin{equation*}\WM{app:01}\begin{pmatrix} \boldsymbol\phi^{(3,0)} \\ \boldsymbol\phi^{(2,1)} \\ \boldsymbol\phi^{(1,2)} \\ \boldsymbol\phi^{(0,3)} \end{pmatrix}\end{equation*}
has only $4$ non-zero columns corresponding to $xf(x,y)$, $yf(x,y)$, $f(x,y)^2$, $f(x,y)$, and a direct computation shows that the determinant of this only non-trivial $4\times 4$ minor is zero anyway. Thus, every generalized Wronskian of $\boldsymbol\phi$ vanishes identically unless it is missing some third-order partial derivative. In particular, there is no non-trivial generalized Wronskian of $\boldsymbol\phi$ using partial derivatives of order at most $3$. Moreover, there are (\emph{a priori} at most) only $4$ non-trivial generalized Wronskians of $\boldsymbol\phi$ using a single partial derivative of order greater than $3$, since it must replace one of the $4$ partial derivatives of order $3$.
\end{remark}

Now, since $f$ is assumed to be generic, its second-order pure derivative $f^{(0,2)}$ is non-zero. Hence, from the vanishing of $W_{2,1}$ and $W_{1,2}$, we obtain that a parametrization of any generic quadratic surface satisfies \eqref{eq:05}. This concludes the first part of the proof.

\begin{remark}
Note that for any generic function $f$, if $W_{2,1}$ and $W_{1,2}$ vanish, then the remaining two generalized Wronskians also vanish. Indeed, we have
\begin{equation}\label{eq:06}
\begin{gathered}
3 f^{(2,0)} W_{3,0} = 2 f^{(1,1)} W_{2,1}-f^{(0,2)} W_{1,2}, \\
3 f^{(0,2)} W_{0,3} = 2 f^{(1,1)} W_{1,2}-f^{(2,0)} W_{2,1},
\end{gathered}
\end{equation}
while both $f^{(2,0)}$ and $f^{(0,2)}$ are non-zero. Furthermore, the same holds for any pair of featured generalized Wronskians except for $W_{3,0}$ and $W_{0,3}$, when the above equations \eqref{eq:06} in variables $W_{2,1}$ and $W_{1,2}$ may turn out to be linearly dependent. This is the case exactly when
$$f^{(0,2)}f^{(2,0)}-4f^{(1,1) 2}=0,$$
which together with $W_{3,0}=0$ and $W_{0,3}=0$ forms a system of partial differential equations. This time, however, apart from parametrizations of certain quadratic surfaces (including degenerate ones), it admits a single family of exotic solutions of the form
$$f(x,y)=\frac{a}{(x+x_0)(y+y_0)}+b_1(x+x_0)+b_2(y+y_0)+c,$$
which arise as parametrizations of certain cubic surfaces. Moreover, note that all these functions are generic, unless $a=0$. Therefore, the choice of equations \eqref{eq:05} was arbitrary only to some extent.
\end{remark}

\begin{remark}
Observe that the last factors of $W_{3,0}$ and $W_{0,3}$ as well as $W_{2,1}$ and $W_{1,2}$ are equivalent up to the order of variables. However, the overall symmetry is broken by the common factor $f^{(0,2) 2}$, which is the result of choosing $\boldsymbol\phi^{(0,4)}$ as a supplementary row. Exactly as we should expect, if we had chosen $\boldsymbol\phi^{(4,0)}$, we would have obtained the same set of generalized Wronskians, but this time with common factor $f^{(2,0) 2}$ instead of $f^{(0,2) 2}$.
\end{remark}

Let
$$Q\colonequals S^{-1}\mathbb R\left[x,y,\partial^{(0,0)},\partial^{(0,1)},\partial^{(1,0)},\partial^{(0,2)},\partial^{(1,1)},\partial^{(2,0)},\partial^{(0,3)},\partial^{(1,2)},\partial^{(0,4)}\right]$$
be the localization of a real polynomial ring in selected $11$ variables at $S$. Since localization commutes with adding new external elements, there is a ring isomorphism
$$S^{-1}R\simeq Q\left[\partial^{(2,1)},\partial^{(3,0)},\partial^{(1,3)},\ldots\right],$$
where the latter is already a polynomial ring over $Q$ in the remaining infinitely many variables. Let us choose a graded lexicographic order on variables $\partial^{(i,j)}$ and then a graded reverse lexicographic order on monomials. We will find a Gr\"obner basis of $\mathfrak a$ with respect to this monomial ordering. For more details on Gr\"obner bases, including definitions and examples, we recommend the reader to go through \cite[ch.~2]{cox2016ideals}.\\

Denote polynomials \eqref{eq:09} respectively by $p_1$, $p_2$. Observe that for every $i,j\geq 0$ and $k=1,2$, ${D_x}^i{D_y}^jp_k$ is linear in the highest-order partial derivatives and thus we can write, e.g.,
\begin{equation}\label{eq:10}\begin{pmatrix}D_xD_xp_1\\D_xD_yp_1\\D_yD_yp_1\\D_xD_xp_2\\D_xD_yp_2\\D_yD_yp_2 \end{pmatrix}=\boldsymbol A_5\begin{pmatrix}\partial^{(0,5)}\\\partial^{(1,4)}\\\partial^{(2,3)}\\\partial^{(3,2)}\\\partial^{(4,1)}\\\partial^{(5,0)}\end{pmatrix}+\boldsymbol b_5,\end{equation}
where
$$\small\boldsymbol A_5\colonequals\begin{pmatrix}
 0 & 0 & 2 \partial^{(1,1)} \partial^{(2,0)} & -3 \partial^{(0,2)} \partial^{(2,0)} & 0 & \partial^{(0,2) 2} \\
 0 & 2 \partial^{(1,1)} \partial^{(2,0)} & -3 \partial^{(0,2)} \partial^{(2,0)} & 0 & \partial^{(0,2) 2} & 0 \\
 2 \partial^{(1,1)} \partial^{(2,0)} & -3 \partial^{(0,2)} \partial^{(2,0)} & 0 & \partial^{(0,2) 2} & 0 & 0 \\
 0 & 0 & \partial^{(2,0) 2} & 0 & -3 \partial^{(0,2)} \partial^{(2,0)} & 2 \partial^{(0,2)} \partial^{(1,1)} \\
 0 & \partial^{(2,0) 2} & 0 & -3 \partial^{(0,2)} \partial^{(2,0)} & 2 \partial^{(0,2)} \partial^{(1,1)} & 0 \\
 \partial^{(2,0) 2} & 0 & -3 \partial^{(0,2)} \partial^{(2,0)} & 2 \partial^{(0,2)} \partial^{(1,1)} & 0 & 0 \\
\end{pmatrix}$$
is a $6\times 6$ matrix and $\boldsymbol b_5$ (the definition of which is irrelevant and therefore has been omitted for brevity) is a $6\times 1$ vector over $S^{-1}R$. Moreover, $\boldsymbol A_5$ and $\boldsymbol b_5$ contain only partial derivatives of order at most $4$. One can easily verify that the determinant of $\boldsymbol A_5$ is equal to
\begin{equation*}\WM{app:02}-64\partial^{(0,2) 3}\partial^{(2,0) 3} \left(\partial^{(0,2)} \partial^{(2,0)}-\partial^{(1,1) 2}\right)^3\end{equation*}
and thus is a unit in $S^{-1}R$. It follows that $\boldsymbol A_5$ is invertible over $S^{-1}R$ and we can rewrite \eqref{eq:10} as
\begin{equation}\label{eq:13}{\boldsymbol A_5}^{-1}\begin{pmatrix}D_xD_xp_1\\D_xD_yp_1\\D_yD_yp_1\\D_xD_xp_2\\D_xD_yp_2\\D_yD_yp_2 \end{pmatrix}=\begin{pmatrix}\partial^{(0,5)}\\\partial^{(1,4)}\\\partial^{(2,3)}\\\partial^{(3,2)}\\\partial^{(4,1)}\\\partial^{(5,0)}\end{pmatrix}+{\boldsymbol A_5}^{-1}\boldsymbol b_5.\end{equation}
Now, the left-hand side is a vector of elements from $\mathfrak a$ and hence so is also the right-hand side. Moreover, since ${\boldsymbol A_5}^{-1}\boldsymbol b_5$ contains only partial derivatives of order at most $4$, the leading term of each polynomial on the right-hand side is a corresponding fifth-order partial derivative. By definition, the ideal $\mathfrak a$ is closed under derivations and thus by differentiating these polynomials, we can obtain an element of $\mathfrak a$ with the leading term being any partial derivative of higher order. Using the same argument, we can likewise write
$${\boldsymbol A_4}^{-1}\begin{pmatrix}D_xp_1\\D_yp_1\\D_xp_2\\D_yp_2\end{pmatrix}=\begin{pmatrix}\partial^{(1,3)}\\\partial^{(2,2)}\\\partial^{(3,1)}\\\partial^{(4,0)}\end{pmatrix}+{\boldsymbol A_4}^{-1}\boldsymbol b_4,\quad{\boldsymbol A_3}^{-1}\begin{pmatrix}p_1\\p_2\end{pmatrix}=\begin{pmatrix}\partial^{(2,1)}\\\partial^{(3,0)}\end{pmatrix}+{\boldsymbol A_3}^{-1}\boldsymbol b_3,$$
since the determinant of
$$\small\boldsymbol A_4\colonequals\begin{pmatrix}
 2 \partial^{(1,1)} \partial^{(2,0)} & -3 \partial^{(0,2)} \partial^{(2,0)} & 0 & \partial^{(0,2) 2} \\
 -3 \partial^{(0,2)} \partial^{(2,0)} & 0 & \partial^{(0,2) 2} & 0 \\
 \partial^{(2,0) 2} & 0 & -3 \partial^{(0,2)} \partial^{(2,0)} & 2 \partial^{(0,2)} \partial^{(1,1)} \\
 0 & -3 \partial^{(0,2)} \partial^{(2,0)} & 2 \partial^{(0,2)} \partial^{(1,1)} & 0 \\
\end{pmatrix}$$
is equal to
\begin{equation*}\WM{app:02}-24\partial^{(0,2) 4}\partial^{(2,0) 2} \left(\partial^{(0,2)} \partial^{(2,0)}-\partial^{(1,1) 2}\right)\end{equation*}
and the determinant of
$$\small\boldsymbol A_3\colonequals\begin{pmatrix}
 0 & \partial^{(0,2) 2} \\
 -3 \partial^{(0,2)} \partial^{(2,0)} & 2 \partial^{(0,2)} \partial^{(1,1)} \\
\end{pmatrix}$$
is equal to
\begin{equation*}\WM{app:02}3\partial^{(0,2) 3} \partial^{(2,0)}.\end{equation*}
Hence all the partial derivatives $\partial^{(2,1)},\partial^{(3,0)},\partial^{(1,3)},\ldots$, which are exactly those not included in the definition of $Q$, are contained in the ideal of leading terms $\langle\mathrm{LT}(\mathfrak a)\rangle$ \cite[Definition~2.5.1]{cox2016ideals}.

\begin{remark}
It is a mere coincidence that after computing second-order partial derivatives of $p_1$ and $p_2$, the number of independent equations is equal to the number of fifth-order partial derivatives of $f$, and thus the matrix $\boldsymbol A_5$ is uniquely determined. The multiplicative monoid $S\subseteq R$ was devised to contain all the prime factors of $\det\boldsymbol A_5$.  However, to obtain $\boldsymbol A_4$ and $\boldsymbol A_3$, we had to arbitrarily choose some subset of variables, and this time it was not a coincidence that both $\det\boldsymbol A_4$ and $\det\boldsymbol A_3$ share the same prime factors as $\det\boldsymbol A_5$. Indeed, there are other choices for which it is no longer the case. Thus, the set of variables to the polynomial ring $Q$ was carefully selected so that both $\boldsymbol A_4$ and $\boldsymbol A_5$ are already invertible in $S^{-1}R$.
\end{remark}

Denote by $G$ the set of polynomials constructed above, such that every monomial $\partial^{(2,1)},\partial^{(3,0)},\partial^{(1,3)},\ldots$ is a leading term $\mathrm{LT}(g)$ of some polynomial $g\in G$. Suppose that $\mathcal Q(\Omega)^\dagger\supsetneq\langle G\rangle$ and let $p\in\mathcal Q(\Omega)^\dagger\setminus\langle G\rangle$. After a complete reduction of $p$ by $G$ we obtain a remainder $r\in\mathcal Q(\Omega)^\dagger\setminus\langle G\rangle$, which is irreducible by $G$, i.e., its leading term $\mathrm{LT}(r)$ is not a multiple of any $\mathrm{LT}(g)$, $g\in G$ \cite[Theorem~2.3.3]{cox2016ideals}. Thus $r$ is an element of the coefficient ring $Q$ and corresponds to some rational function in selected $11$ variables. We will prove that $r=0$, which eventually will give us the desired contradiction. By definition, it vanishes for any tuple consisting of $x$, $y$, and relevant partial derivatives of some function parametrizing a generic quadratic surface at $(x,y)$. Since $r$ is rational, it is enough to show that the set of such arguments has a non-empty interior as a subset of $\mathbb R^{11}$.\\

For this, we define an implicit function $\boldsymbol\psi:\mathbb R^{11}\to\mathbb R^{11}$ in the following way. Let $f$ be a parametrization of some quadratic surface satisfying \eqref{eq:38} with $a_{33}=1$. Then $\boldsymbol\psi$ maps the tuple of parameters
\begin{equation}\label{eq:03}\begin{pmatrix} x, & y, & a_{11}, & a_{12}, & a_{13}, & a_{22}, & a_{23}, & b_1, & b_2, & b_3, & c \end{pmatrix}\end{equation}
to the tuple
$$\begin{pmatrix} x, & y, & f, & f^{(0,1)}, & f^{(1,0)}, & f^{(0,2)}, & f^{(1,1)}, & f^{(2,0)}, & f^{(0,3)}, & f^{(1,2)}, & f^{(0,4)} \end{pmatrix}$$
consisting of $x$, $y$ and relevant partial derivatives of $f$ at $(x,y)$. We can obtain an explicit formula for $\boldsymbol\psi$ by symbolically solving the quadratic equation \eqref{eq:38} first and then symbolically differentiating the result. Since in general there are two possible solutions for $f$, we have to locally select an arbitrary branch of the square root function, so that $\boldsymbol\psi$ is smooth.\\

Now, consider a generic function
$$f(x,y)=\sqrt{1+x^2+y^2}$$
parametrizing a quadratic surface represented by the tuple of parameters
\begin{equation}\label{eq:11}\begin{pmatrix} x, & y, & -1, & 0, & 0, & -1, & 0, & 0, & 0, & 0, & -1 \end{pmatrix}.\end{equation}
Since \eqref{eq:12} depends continuously on \eqref{eq:03}, any point in some open neighborhood $U$ of \eqref{eq:11} also corresponds to a parametrization of some generic quadratic surface and thus $r\circ\boldsymbol\psi$ vanishes on $U$. Hence, it is enough to show that $\boldsymbol\psi(U)$ has a non-empty interior. Computing the Jacobian determinant of $\boldsymbol\psi$ at \eqref{eq:11} yields
\begin{equation*}\WM{app:03}\frac{9 \left(x^2+1\right)^4}{128 \left(x^2+y^2+1\right)^{11}},\end{equation*}
which is non-zero. Hence $\boldsymbol\psi$ is a local diffeomorphism, and so there exists an open subset $V\subseteq U$ such that $\boldsymbol\psi\vert_V:V\to\boldsymbol\psi(V)$ is a diffeomorphism. In particular, $\boldsymbol\psi(V)$ is open. However, recall that it is contained in the zero set of $r$, which must therefore be a zero function, a contradiction. It follows that $\mathfrak a\subseteq\mathcal Q(\Omega)^\dagger=\langle G\rangle\subseteq\mathfrak a$, which means that $\mathcal Q(\Omega)^\dagger=\mathfrak a$ and moreover $G$ is, in fact, a reduced Gr\"obner basis of $\mathfrak a$, which concludes the proof.
\end{proof}

\begin{remark}
Now, we are able to clarify in what sense equations \eqref{eq:05} are minimal. Namely, \eqref{eq:09} form a reduced Gr\"obner basis of $\mathcal Q(\Omega)^\dagger$, while, as it will turn out, we would obtain the same results as in \thref{thm:01} and \thref{thm:02} for any generating set of $\mathcal Q(\Omega)^\dagger$. Although the elements \eqref{eq:09} seem to be the best choice, the reduced Gr\"obner basis is by no means unique. Besides, with \thref{pro:01} at hand, finding other generating sets becomes a purely algorithmic task.
\end{remark}

\section{Smoothing properties and their connection with holomorphicity}

At some point in the future, we will want to deduce a linear dependence of a set of functions from the vanishing of their generalized Wronskians. For this, we will use the main result from \cite{WOLSSON198973}, where the necessary and sufficient conditions are established. Although the author roughly requires that all the generalized Wronskians must vanish, in the course of the inherently constructive proof, he considers only finitely many generalized Wronskians of bounded order. Nevertheless, our initial assumption that the function $f$ is merely an element of $W^{3,1}_{\mathrm{loc}}(\Omega)$ is too weak for any non-trivial generalized Wronskian to be well-defined. Therefore, we need to somehow improve the smoothness of $f$. As it turns out, the differentiability of class $C^5$ will be sufficient, and thus we will not use the following fact in its full generality.

\begin{lemma}\thlabel{lem:03}
Let $f\in W^{3,1}_{\mathrm{loc}}(\Omega)$ be a function defined on a connected open subset $\Omega\subseteq\mathbb R^2$. Suppose that $f$ is generic and is a weak solution to the system of partial differential equations \eqref{eq:05}. Then $f$ is infinitely differentiable.
\end{lemma}

\begin{proof}
Let $u,v:\Omega\to\mathbb R$ be the functions defined as follows:
\begin{equation}\label{eq:02}
\begin{aligned}
u(x,y) &\colonequals\frac{f^{(2,0)}(x,y)-f^{(0,2)}(x,y)}{\left|f^{(0,2)}(x,y) f^{(2,0)}(x,y)-f^{(1,1)}(x,y)^2\right|^{3/4}} \\
v(x,y) &\colonequals\frac{2 f^{(1,1)}(x,y)}{\left|f^{(0,2)}(x,y) f^{(2,0)}(x,y)-f^{(1,1)}(x,y)^2\right|^{3/4}}.
\end{aligned}
\end{equation}
Note that they are well-defined by the assumption that $f$ is generic. Since they depend only on the second-order partial derivatives of $f$, which are assumed to be elements of $W^{1,1}_{\mathrm{loc}}(\Omega)$, and moreover the Hessian determinant of $f$ is locally bounded away from $0$, both functions $u,v$ are elements of $W^{1,1}_{\mathrm{loc}}(\Omega)$. Computing their weak partial derivatives and applying \eqref{eq:05}, one can find out that they satisfy the Cauchy-Riemann equations:
\begin{equation*}\WM{app:04}\begin{aligned}
u^{(1,0)}(x,y)-v^{(0,1)}(x,y)&=\pm\frac{\left(3 f^{(0,2)}(x,y)+f^{(2,0)}(x,y)\right) p_1-2 f^{(1,1)}(x,y) p_2}{4 f^{(0,2)}(x,y) \left|f^{(0,2)}(x,y) f^{(2,0)}(x,y)-f^{(1,1)}(x,y)^2\right|^{7/4}}=0,\\
u^{(0,1)}(x,y)+v^{(1,0)}(x,y)&=\pm\frac{2 f^{(1,1)}(x,y) p_1-\left(3 f^{(2,0)}(x,y)+f^{(0,2)}(x,y)\right) p_2}{4 f^{(2,0)}(x,y) \left|f^{(0,2)}(x,y) f^{(2,0)}(x,y)-f^{(1,1)}(x,y)^2\right|^{7/4}}=0,
\end{aligned}\end{equation*}
where $p_1,p_2$ denote respectively the left-hand sides of \eqref{eq:05} and $\pm$ is the sign of the Hessian determinant. Again, the above formulas are well-defined by the assumption that $f$ is generic. Thus $u,v$ are analytic on $\Omega$ \cite[Theorem~9]{10.2307/2321164}. This is actually a special case of a more general result on the regularity of solutions of hypo-elliptic partial differential equations \cite{Hoermander1961}.\\

Now, observe that the first-order partial derivatives of $u,v$ as well as the left-hand sides of \eqref{eq:05} are linear in third-order partial derivatives of $f$ and thus we can write, e.g.,
$$\begin{pmatrix}u^{(1,0)}\\u^{(0,1)}\\p_1\\p_2\end{pmatrix}=\boldsymbol A\begin{pmatrix}f^{(0,3)}\\f^{(1,2)}\\f^{(2,1)}\\f^{(3,0)}\end{pmatrix},$$
which allows us to express all the third-order partial derivatives of $f$ in terms of the first-order partial derivatives of $u$ and the second-order partial derivatives of $f$. To do this, we only need to verify that the matrix $\boldsymbol A$ is invertible. Indeed, its determinant is equal to
\begin{equation*}\WM{app:04}\pm\frac{4 f^{(0,2)}(x,y) f^{(2,0)}(x,y)}{\left|f^{(0,2)}(x,y) f^{(2,0)}(x,y)-f^{(1,1)}(x,y)^2\right|^{1/2}},\end{equation*}
where $\pm$ is the sign of the Hessian determinant. Therefore we have
\begin{equation}\label{eq:14}\begin{pmatrix}f^{(0,3)}\\f^{(1,2)}\\f^{(2,1)}\\f^{(3,0)}\end{pmatrix}=\boldsymbol A^{-1}\begin{pmatrix}u^{(1,0)}\\u^{(0,1)}\\p_1\\p_2 \end{pmatrix}=\boldsymbol A^{-1}\begin{pmatrix}u^{(1,0)}\\u^{(0,1)}\\0\\0 \end{pmatrix},\end{equation}
where the right-hand side is linear in first-order partial derivatives of $u$ and algebraic in second-order partial derivatives of $f$. Since $\boldsymbol A^{-1}=(\det\boldsymbol A)^{-1}(\adj\boldsymbol A)$, where $\det\boldsymbol A$ is locally bounded away from $0$ and $\adj\boldsymbol A$ is a polynomial in second-order partial derivatives of $f$, we have $\boldsymbol A^{-1}\in W^{1,1}_{\mathrm{loc}}(\Omega)$. It follows that the left-hand side (and hence also the right-hand side) is an element of $W^{1,1}_{\mathrm{loc}}(\Omega)$, which means by the very definition that $f\in W^{4,1}_{\mathrm{loc}}(\Omega)$. Now we are able to weakly differentiate the equalities \eqref{eq:14} and iterate the same argument to see that $f$ is indeed infinitely differentiable on $\Omega$. This ends the proof.
\end{proof}

\section{Proofs of the main theorems}

Before we move on to the essential part of this section, we will prove the following lemma, which will play a key role in the proofs of both main theorems:

\begin{lemma}\thlabel{lem:02}
Let $f\in W^{3,1}_{\mathrm{loc}}(\Omega)$ be a function defined on a connected open subset $\Omega\subseteq\mathbb R^2$. Suppose that $f$ is generic. Then $f$ satisfies the system of partial differential equations \eqref{eq:05} if and only if its graph is contained in a quadratic surface.
\end{lemma}

\begin{proof}
Left implication $(\impliedby)$ follows immediately from \thref{pro:01}. A proof of the right implication $(\implies)$ is not so straightforward, since we want to deduce a linear dependence of a set of functions from the vanishing of their generalized Wronskians, which fails to be true in general \cite{Peano} and therefore needs specific arguments.\\

Again, we adopt the notation from \cite{WOLSSON198973}.

\begin{definition}[{\cite[Definition~2]{WOLSSON198973}}]
A \emph{critical point} of $\boldsymbol\phi$ is a point of the domain at which all generalized Wronskians of $\boldsymbol\phi$ vanish.
\end{definition}

\begin{definition}[{\cite[Definition~3]{WOLSSON198973}}]
An $r\times r$ \emph{generalised sub-Wronskian} of $\boldsymbol\phi$, $1\leq r\leq n$, is a generalised Wronskian of any subsequence of $\boldsymbol\phi$.
\end{definition}

Note that not every minor of a generalized Wronskian is a generalized sub-Wronskian. Indeed, the above definition requires that the minor also satisfy the additional condition for orders of partial derivatives.

\begin{definition}[{\cite[Definition~4]{WOLSSON198973}}]
The \emph{order} of a critical point $t$ of $\boldsymbol\phi$ is the largest positive integer $r$ for which some $r\times r$ generalized sub-Wronskian of $\boldsymbol\phi$ is not zero at $t$. Should all sub-Wronskians vanish at $t$, the order is defined to be zero.
\end{definition}

We will show that every $t\in\Omega$ is a critical point of $\boldsymbol\phi$ of order $9$. First, observe that all $10\times 10$ generalized Wronskians of $\boldsymbol\phi$ vanish identically on $\Omega$. Indeed, from \thref{lem:03} we infer that $f$ is smooth and thus all its generalized Wronskians are well-defined. Moreover, by \thref{ass:01} they belong to $\mathcal Q(\Omega)^\dagger$ and hence they vanish identically on $\Omega$ since both generators of $\mathcal Q(\Omega)^\dagger$ do.\\

We are left to prove that for every $t\in\Omega$ there exists a $9\times 9$ generalized sub-Wronskian of $\boldsymbol\phi$ that is non-zero at $t$. Observe that, i.a., every $9\times 9$ minor of the following $9\times 10$ matrix
$$\boldsymbol W\colonequals\begin{pmatrix} \boldsymbol\phi \\ \boldsymbol\phi^{(0,1)} \\ \boldsymbol\phi^{(1,0)} \\ \boldsymbol\phi^{(0,2)} \\ \boldsymbol\phi^{(1,1)} \\ \boldsymbol\phi^{(2,0)} \\ \boldsymbol\phi^{(0,3)} \\ \boldsymbol\phi^{(1,2)} \\ \boldsymbol\phi^{(0,4)} \end{pmatrix}$$
that comprises the first row is a valid sub-Wronskian of $\boldsymbol\phi$ and thus it is enough to show that $\boldsymbol W=\boldsymbol W(t)$ has full rank at every $t\in\Omega$. Denote by $W_i$ the minor of $\boldsymbol W$, obtained by deleting the \nth{i} column, and suppose that all $W_i$ are zero. A direct computation shows that
\begin{equation*}\WM{app:01}W_6=4 f^{(0,2)} \left(3 f^{(0,2)} f^{(0,4)}-4 f^{(0,3) 2}\right),\end{equation*}
which implies
$$f^{(0,4)}=\frac{4 f^{(0,3) 2}}{3 f^{(0,2)}}.$$
Applying the above result to the definition of $W_5$ yields
\begin{equation*}\WM{app:01}W_5=-24 f^{(0,2) 3} f^{(0,3)}\end{equation*}
and consequently
$$f^{(0,3)}=0.$$
It follows that
\begin{equation*}\WM{app:01}W_4=36 f^{(0,2) 5},\end{equation*}
which, finally, gives us the desired contradiction.\\

We are now at a point where we can apply the following fundamental theorem:

\begin{lemma}[{\cite[Theorem~2]{WOLSSON198973}}]\thlabel{lem:01}
If $G$ is an open connected set consisting of critical points of the same order $r>0$, then $\boldsymbol\phi$ has a linearly independent subset $S_r=\{\phi_1,\ldots,\phi_r\}$, say, which is a basis of $\mathrm{span}(\boldsymbol\phi)$, and consequently $\boldsymbol\phi$ is linearly dependent on $G$.
\end{lemma}

By \thref{lem:01} we know that $\boldsymbol\phi$ is linearly dependent on $\Omega$. This concludes the proof.
\end{proof}

\begin{remark}
Observe that the system of partial differential equations \eqref{eq:05} is satisfied if and only if a pair of functions \eqref{eq:02} satisfies the Cauchy-Riemann equations. Thus, the graph of $f$ is contained in a quadratic surface if and only if $u+iv$ is holomorphic. Moreover, if $f$ satisfies \eqref{eq:38}, then a direct computation shows that $u+iv$ is simply a quadratic polynomial:
\begin{equation}\label{eq:15}\WM{app:04}u+iv=\frac{((Q_{1,1}-Q_{2,2})-2iQ_{1,2})+2(Q_{1,4}-iQ_{2,4})z+Q_{4,4}z^2}{|\det\boldsymbol Q|^{3/4}},\end{equation}
where
\begin{equation}\label{eq:04}\boldsymbol Q\colonequals\begin{pmatrix}a_{11} & \frac{1}{2}a_{12} & \frac{1}{2}a_{13} & \frac{1}{2}b_1 \\[.3em] \frac{1}{2}a_{12} & a_{22} & \frac{1}{2}a_{23} & \frac{1}{2}b_2 \\[.3em] \frac{1}{2}a_{13} & \frac{1}{2}a_{23} & a_{33} & \frac{1}{2}b_{3} \\[.3em] \frac{1}{2}b_1 & \frac{1}{2}b_2 & \frac{1}{2}b_3 & c \end{pmatrix}\end{equation}
is a symmetric matrix defining an affine quadratic form \eqref{eq:38} and $Q_{i,j}$ is the $(i,j)$ minor of $\boldsymbol Q$, i.e., the determinant of the submatrix formed by deleting the \nth{i} row and \nth{j} column.
\end{remark}

\begin{remark}
Since a quadratic surface is uniquely determined by $9$ parameters and the quadratic polynomial \eqref{eq:15} has only $5$ parameters, a natural question arises: Which functions correspond to the same quadratic polynomial? Note that $u,v$ depend only on second-order partial derivatives of $f$, which means that adding linear terms does not change \eqref{eq:15}. For completeness, we still need one more parameter. Careful inspection of \eqref{eq:15} shows, e.g., that every function of the form
$$f(x,y)\colonequals a\sqrt{1-ax^2-ay^2}+bx+cy+d$$
gives rise to the same quadratic polynomial $u+iv=z^2$. Unfortunately, the general answer is far more complicated and will not be given here.
\end{remark}

Finally, we still will need one more simple fact, which can be verified by a direct computation:

\begin{assertion}[{\cite[Theorem~1]{2203.04082}}]\thlabel{ass:02}
Let $f:\mathbb R^2\supset\Omega\to\mathbb R$ be a function of class $C^2$ defined on an open subset of $\mathbb R^2$ and satisfying a quadratic equation \eqref{eq:38}. Then the following formula holds:
\begin{equation*}\WM{app:04}H_f(x,y)\cdot\Delta_f(x,y)^2=-16\det\boldsymbol Q,\end{equation*}
where $H_f$ is the Hessian determinant of $f$, $\Delta_f$ is the discriminant of \eqref{eq:38} with respect to the variable $f$ and $\boldsymbol Q$ is just as defined in \eqref{eq:04}.\hfill\ensuremath\blacksquare
\end{assertion}

Now we are ready to prove \thref{thm:01}.

\begin{proof}[Proof of Theorem~$\ref{thm:01}$]
Define $U\subseteq\Omega$ to be the subset consisting of those points, where the Hessian determinant of $f$ is positive. Note that $U$ is open, which immediately follows from the continuity of partial derivatives. Moreover, the assumption on $f$ asserts that it is also non-empty.\\

First, we will show the right implication $(\implies)$. Since $f\vert_U$ is generic, by \thref{lem:02} its graph is contained in a quadratic surface. Let $t\in\Omega$ be a limit point of $U$, i.e., such that there exists a sequence $(t_n)$ of points in $U$ whose limit is $t$. Now, if $\Delta_f$ vanishes identically on $U$, then $f$ is affine, a contradiction. Thus there exists $u\in U$ such that $\Delta_f(u)>0$, which implies $-16\det\boldsymbol Q>0$. Moreover, since $\Delta_f$ is a quadratic polynomial, the sequence $\Delta_f(t_n)^2$ is bounded from above. It follows that $H_f(t_n)=-16\det\boldsymbol Q\cdot\Delta_f(t_n)^{-2}$ is bounded from below by some positive constant $\varepsilon>0$. In particular, we have $H_f(t)\geq\varepsilon>0$ and hence $t\in U$, by the very definition. Thus, we have shown that $U$ contains all its limit points, which makes it closed in $\Omega$. However, recall that $U$ is also open, in which case we have simply $U=\Omega$. This concludes the first part of the proof.\\

The remaining left implication $(\impliedby)$ follows right the same way. Since the graph of $f\vert_U$ is by assumption contained in a quadratic surface, we repeat the above limit point argument to see that likewise $U=\Omega$. With this result at hand, we once again apply \thref{lem:02} to conclude the proof.
\end{proof}

Finally, we are in a position to drop the assumption on the Hessian determinant. It turns out to be important, since \eqref{eq:05} is also satisfied by parametrizations of some ruled surfaces, the Hessian determinant of which is non-positive.

\begin{proof}[Proof of Theorem~$\ref{thm:02}$]
Define the following open sets:
\begin{align*}
\Omega_{\mathrm a}&\colonequals\left\{f^{(0,2)}\neq 0, f^{(2,0)}\neq 0, f^{(0,2)} f^{(2,0)}-f^{(1,1) 2}<0\right\},\\
\Omega_{\mathrm b}&\colonequals\Int\left\{f^{(0,2)} f^{(2,0)}-f^{(1,1) 2}=0\right\},\\
\Omega_{\mathrm c}&\colonequals\Int\left\{f^{(2,0)}=0\right\},\\
\Omega_{\mathrm d}&\colonequals\Int\left\{f^{(0,2)}=0\right\}.
\end{align*}
Clearly their sum $\Omega_{\mathrm a}\cup\Omega_{\mathrm b}\cup\Omega_{\mathrm c}\cup\Omega_{\mathrm d}$ is dense in $\Omega$. By definition, on each connected component of $\Omega_{\mathrm b}$, the graph of $f$ is contained in a developable surface. Moreover, on each connected component of $\Omega_{\mathrm c}$ (respectively: $\Omega_{\mathrm d}$) $f$ is linear along every straight line parallel to the $OX$ (respectively: $OY$) axis and thus, again by definition, its graph is contained in a Catalan surface with directrix plane $XZ$ (respectively: $YZ$). Hence, to prove the right implication $(\implies)$, it remains for us to show that on every connected component of $\Omega_{\mathrm a}$ the graph of $f$ satisfying \eqref{eq:05} is contained in a doubly-ruled surface, which readily follows from \thref{lem:02}. Indeed, we immediately obtain that the graph of $f$ is contained in a quadratic surface of negative Gaussian curvature. The only two are hyperbolic paraboloid and single-sheeted hyperboloid, both of which are doubly-ruled \cite[p.~15]{hilbert1999geometry}. This concludes the first part of the proof.\\

On the other hand, observe that $f\vert_{\Omega_{\mathrm b}}$ automatically satisfies \eqref{eq:05}. Indeed, denote
$$H_f(x,y)\colonequals f^{(0,2)}(x,y) f^{(2,0)}(x,y)-f^{(1,1)}(x,y)^2$$
and observe that
\begin{align*}
p_1 &= -4 f^{(1,2)} H_f+f^{(0,2)} \frac{\partial H_f}{\partial x}+2 f^{(1,1)} \frac{\partial H_f}{\partial y} = 0, \\
p_2 &= -4 f^{(2,1)} H_f+f^{(2,0)} \frac{\partial H_f}{\partial y}+2 f^{(1,1)} \frac{\partial H_f}{\partial x} = 0,
\end{align*}
where $p_1,p_2$ again stand for the left-hand sides of \eqref{eq:05}, respectively. Moreover, $f\vert_{\Omega_{\mathrm c}}$ (respectively: $f\vert_{\Omega_{\mathrm d}}$) satisfies $f^{(2,0)}\equiv 0$ and consequently $f^{(3,0)}\equiv 0$ (respectively: $f^{(0,2)}\equiv 0$ and consequently $f^{(0,3)}\equiv 0$), in which case a simple check shows that it satisfies \eqref{eq:05} as well. Hence to prove the left implication $(\impliedby)$ it remains for us to show that for each $\Omega_i$, $f\vert_{\Omega_i\cap\Omega_{\mathrm a}}$ satisfies \eqref{eq:05}. However, $\Omega_i\cap\Omega_{\mathrm a}$ is non-empty if and only if the graph of $f\vert_{\Omega_i}$ is contained in a doubly-ruled surface of negative Gaussian curvature. The only two are hyperbolic paraboloid and single-sheeted hyperboloid \cite[p.~15]{hilbert1999geometry}, both of which are quadratic. The assertion follows from \thref{lem:02}, which concludes the proof.
\end{proof}

\begin{remark}
Denote by
$$f(t+\boldsymbol u)\equalscolon\sum_{k=0}^3\frac{1}{k!}f_k(t)[\boldsymbol u]+o(\|\boldsymbol u\|^3)$$
the series expansion of $f$ at $t\in\Omega$, where $f_k(t)$ stands for its homogeneous component of degree $k$. Generally, any second-order homogeneous polynomial vanishes on at most two lines in $\mathbb{RP}^1$. Therefore, if the graph of $f$ is contained in a doubly-ruled surface, $f_2(t)$ vanishes exactly on the two rulings that pass through $t$. Moreover, $f_3(t)$ likewise must vanish on the same two rulings. In particular, it follows that $f_2(t)$ divides $f_3(t)$ as a polynomial. So it should come as no surprise to us that, for a generic function $f$, equations \eqref{eq:05} are satisfied if and only if $f_2(t)$ divides $f_3(t)$. Indeed,
$$f_3(t)[\boldsymbol u]=\left(\frac{f^{(3,0)}(t)}{f^{(2,0)}(t)}u_1+\frac{f^{(0,3)}(t)}{f^{(0,2)}(t)}u_2\right)f_2(t)[\boldsymbol u]-\left(\frac{p_2(t)}{f^{(2,0)}(t)f^{(0,2)}(t)}u_1+\frac{p_1(t)}{f^{(2,0)}(t)f^{(0,2)}(t)}u_2\right)u_1u_2,$$
where $p_1,p_2$ again stand for the left-hand sides of \eqref{eq:05}, respectively. Observe that the remainder is a product of $u_1$, $u_2$, and some linear homogeneous polynomial, whereas $f_2(t)[\boldsymbol u]$ is divisible by neither $u_1$ nor $u_2$, which means it can not divide the remainder unless the latter is zero. Thus, we have found yet another way of looking at \eqref{eq:05}. For, in the generic case, it arises as generalized Wronskians of a certain set of functions, as Cauchy-Riemann equations for a certain pair of functions, and now as coefficients of a certain remainder from dividing $f_3(t)$ by $f_2(t)$.
\end{remark}

\begin{remark}
Since the proofs of both theorems were mainly algebraic, the same results hold also in a complex setting, if we assume $f:\mathbb C^2\supseteq\Omega\to\mathbb C$ to be holomorphic. Although the author did not point it out, the same applies likewise to the cited work \cite{WOLSSON198973} concerning generalized Wronskians, which allows us to apply the results in an analogous manner. Only smoothing \thref{lem:03} ceases to make sense, but actually it is not needed anymore.
\end{remark}

Let us conclude our considerations with an improvement of a well-known corollary from the aforementioned theorem of Maschke-Pick-Berwald \cite[Theorem~4.5]{nomizu1994affine}:

\begin{corollary}\thlabel{cor:01}
Let $S\subset\mathbb R^n$, $n\geq 3$, be a convex hypersurface of class $C^{2,1}$ such that for almost every $x\in S$ there is a quadratic hypersurface having third-order contact with $S$ at $x$. Then $S$ is itself a quadratic hypersurface.
\end{corollary}

\begin{proof}[Proof of \thref{cor:01} for $n=3$]
Define $f:\mathbb R^2\supseteq U\to\mathbb R$ to be a function such that its graph contains an open subset of $S$. Fix $x\in U$ such that $f$ is thrice differentiable at $x$ and define $g:\mathbb R^2\supseteq V\to\mathbb R$ to be a parametrization of a quadratic surface having third-order contact with $S$ at $x$. It follows from the \enquote{if} part of \thref{thm:01} that $g$ satisfies \eqref{eq:05} at $x$. Moreover, by assumption, we have the equality of jets $J_x^3g=J_x^3f$ and hence $f$ likewise satisfies \eqref{eq:05} at $x$. Now, since $x$ was arbitrary, it means that $f$ satisfies \eqref{eq:05} almost everywhere on the domain $U$, and finally, from the \enquote{only if} part of \thref{thm:01}, we obtain that its graph is contained in a quadratic surface. This concludes the proof.
\end{proof}

\begin{proof}[Proof of \thref{cor:01} for $n\geq 4$]
Let $H$ be any $3$-dimensional affine flat that intersects $S$, and consider the foliation of $\mathbb R^n$ with $3$-dimensional affine flats parallel to $H$. By Fubini's theorem, almost all of the leaves satisfy the assumption of \thref{cor:01} for $n=3$, whence $S\cap H$ can be approached by a sequence of quadratic surfaces $S\cap H_n$, and thus it must be a quadratic surface itself. Now, since $H$ was arbitrary, it follows that every section of $S$ by a $3$-dimensional affine flat, and thus also by a $2$-dimensional affine flat, is contained in some quadric.\\

Finally, let $x\in S$ be any point on the hypersurface, let $p,q\in S$ be other two points on the hypersurface that are sufficiently close to $x$, and let $Q$ be the unique quadratic hypersurface having first-order contact with $S$ at $p$ and second-order contact with $S$ at $q$. Every section of $S$ by a $2$-dimensional flat $H$ passing through $p,q$ is contained in the unique quadratic curve having first-order contact with $S$ at $p$ and second-order contact with $S$ at $q$, which is necessarily a section of $Q$. Hence $S\cap H\subseteq Q\cap H$ for every $2$-dimensional flat $H$ passing through $p,q$. Since such flats foliate a neighbourhood of $x$, $S$ is locally a quadratic hypersurface, coinciding with $Q$. Because $x$ was arbitrary, this concludes the proof.
\end{proof}

The above differential characterization of quadratic hypersurfaces is expressed in the language of differential geometry rather than differential equations. Unlike \thref{thm:01}, the assumption of \thref{cor:01} is clearly invariant under an affine change of coordinate system. Importantly, it requires only that the hypersurface be of class $C^{2,1}$, whereas the theorem of Maschke-Pick-Berwald, as formulated in \cite[Theorem~4.5]{nomizu1994affine}, requires the hypersurface to be smooth, and its proof indeed relies on higher-order tensors.

\subsection*{Data availability}

The data that support the findings of this study are available from a public institutional repository, \url{https://drive.google.com/drive/folders/1-pCxC69RI0g10bmWQdHiZ26d3faswKYX?usp=sharing}.

\section*{Acknowledgements}

I would like to thank my doctoral advisor, Prof. Michał Wojciechowski, for his breakthrough ideas, insistence on giving a profound background to all the intermediate results, help in completing all the details, and the particularly high amount of effort spent on editorial work.

\begin{appendix}
\section{Computer assistance in symbolic computations}\label{app:00}
All functions were implemented in \textsf{Wolfram Mathematica} \cite{Mathematica}. The code itself is available from a public institutional repository, \url{https://drive.google.com/drive/folders/1-pCxC69RI0g10bmWQdHiZ26d3faswKYX?usp=sharing}. Computations were performed on a \textsf{Linux x86 (64-bit)} machine with a single \textsf{Intel® Xeon® CPU E5-2697 v3} processor and \textsf{64GB} memory. The total execution time was negligible.

\subsection{\fox~\textsf{Notebook-1.nb}}\label{app:01}
In the beginning, we use symbolic differentiation \textsf{D} to obtain the Wronskian matrix of \eqref{eq:01}. Afterward, we use \textsf{Minors} to compute $210$ symbolic determinants of order $4$ and thus find out that the four rows corresponding to third-order partial derivatives are indeed linearly dependent. Then we again use \textsf{Minors} to compute $11$ symbolic determinants of order $10$, among which the only non-trivial ones are $W_{3,0}$, $W_{2,1}$, $W_{1,2}$ and $W_{0,3}$. Finally, we use \textsf{Minors} to compute $10$ symbolic determinants of order $9$ and select the simplest-looking ones. Based on them, we solve some simple linear equations to find out that all the featured minors can not vanish simultaneously. Performing the same calculations with pen and paper would be tedious, but possible. Although we sometimes applied \textsf{Factor} to factorize the results, in all cases the factorization turned out to be trivial.

\subsection{\fox~\textsf{Notebook-2.nb}}\label{app:02}
In the beginning, we use symbolic differentiation \textsf{D} to compute the left-hand side of \eqref{eq:10}. Afterward, we use \textsf{CoefficientArrays} to extract the explicit form of the matrix $\boldsymbol A_5$ of order $6$. Then we apply \textsf{Det} and \textsf{Factor} to obtain its determinant in the form from which we can readily see that it is an element of the multiplicative submonoid $S$. Finally, we repeat the same steps for $\boldsymbol A_4$ of order $5$ and $\boldsymbol A_3$ of order $4$. The same calculations could well be done with pen and paper, though it is pointless.

\subsection{\fox~\textsf{Notebook-3.nb}}\label{app:03}
At the beginning, we solve the quadratic equation \eqref{eq:38} for $f$, assuming previously that $a_{33}=1$. Afterward, we use symbolic differentiation to obtain the explicit formula for $\boldsymbol\psi$. Then we use \textsf{Grad} to compute the Jacobian matrix of $\boldsymbol\psi$ with respect to the $11$-dimensional vector \eqref{eq:03}. Now, since its symbolic determinant is difficult to compute even for a supercomputer, we instantiate the matrix at \eqref{eq:11} and only then do we apply \textsf{Det} and \textsf{Factor} to obtain its determinant of order $11$ in the simplest form. The content of this notebook is by far the most demanding computational task because, in addition to the heavy workload, it also requires manipulating algebraic expressions containing square roots.

\subsection{\fox~\textsf{Notebook-4.nb}}\label{app:04}
In the beginning, we define $p_1$, $p_2$, $u$, $v$, and verify that $u$, $v$ satisfy \eqref{eq:02}, which requires symbolic differentiation and manipulating algebraic expressions containing square roots. Afterward, we use \textsf{CoefficientArrays} to extract the explicit form of the matrix $\boldsymbol A$ of order $4$. Then we apply \textsf{Det} and \textsf{Together} to obtain its determinant in the simplest form. Further, we solve the quadratic equation \eqref{eq:38} for $f$ and then put the result into the formula for $u+iv$. We apply \textsf{Together} and \textsf{PowerExpand} to bring the result to a simpler form. Finally, we define the matrix $\boldsymbol Q$ of order $4$ and verify the formula \eqref{eq:15}, using \textsf{Minors} to compute $16$ symbolic determinants of order $3$ along the way. Then we apply \textsf{Together} to force the expansion of the underlying expression. At the very end, we verify \thref{ass:02}, using symbolic differentiation \textsf{D} composed with \textsf{Det} to obtain the Hessian determinant of $f$ and \textsf{Discriminant} to compute the discriminant of \eqref{eq:38} with respect to the variable $f$. Again, we apply \textsf{Together} to force the expansion of the underlying expression. The same calculations could well be done with pen and paper, though it is pointless.
\end{appendix}

\bibliography{references}{}
\bibliographystyle{amsplain}

\end{document}